\documentclass[11pt]{article}
\usepackage{tikz}
\usepackage{tikz-cd}

\usepackage{charter}

\usepackage[tmargin=2.5cm,rmargin=3.5cm,lmargin=3.5cm]{geometry}

\usepackage[english]{babel}
\usepackage[T1]{fontenc}
\usepackage[utf8]{inputenc}

\usepackage{verbatim}
\usepackage{graphicx}

\usepackage{xcolor}

\usepackage{makeidx}

\usepackage{dsfont}
\usepackage{mathrsfs}
\usepackage{amsmath}
\usepackage{amsthm}
\usepackage{upref}
\usepackage{amsfonts}
\usepackage{amssymb}
\usepackage{xy}
\xyoption{all}
\usepackage{pb-diagram,pb-xy}

\usepackage{diagbox}

\usepackage[colorlinks=true, linkcolor=violet]{hyperref}

\usepackage{pifont}% http://ctan.org/pkg/pifont
\newcommand{\cmark}{\ding{51}}%
\newcommand{\xmark}{\ding{55}}%

\newcommand{\natbar}{\overline{\mathbb{N}}}
\newcommand{\nat}{\mathbb{N}}
\newcommand{\rational}{\mathbb{Q}}
\newcommand{\real}{\mathbb{R}}
\newcommand{\obase}[1]{\overrightarrow{\mathfrak{B}}^{{#1}}}
\newcommand{\dQcircle}{\overrightarrow S^{\hspace{-0.6mm}1}_{\hspace{-1mm}\scriptscriptstyle\rational}}
\newcommand{\Qcircle}{S^1_{\scriptscriptstyle\rational}}
\newcommand{\tcircle}{S^1}
\newcommand{\dcircle}{\overrightarrow {S^1}}
\newcommand{\LPO}{\mathbf{Lpo}}	
\newcommand{\Top}{\mathbf{Top}}
\newcommand{\ie}{i.e.}	

\newcommand{\dreal}{\overset{\mathbf{\to}}{\real}}
\newcommand\Ecal{\mathcal E}

\newcommand{\pyccmd}[1]{\overrightarrow{\mathcal{O}}({#1})}
	
\makeatletter
\DeclareFontFamily{U}{tipa}{}
\DeclareFontShape{U}{tipa}{m}{n}{<->tipa10}{}
\newcommand{\arcc@char}{{\usefont{U}{tipa}{m}{n}\symbol{62}}}%

\newcommand{\arcc}[1]{\mathpalette\arcc@arc{#1}}

\newcommand{\arcc@arc}[2]{%
  \sbox0{$\m@th#1#2$}%
  \vbox{
    \hbox{\resizebox{\wd0}{\height}{\arcc@char}}
    \nointerlineskip
    \box0
  }%
}
\makeatother

\theoremstyle{plain}
\newtheorem{theo}{Theorem}[section]

\theoremstyle{definition}
\newtheorem{prop}[theo]{Proposition}

\theoremstyle{definition}
\newtheorem{corol}[theo]{Corollary}

\theoremstyle{definition}
\newtheorem{lem}[theo]{Lemma}

\theoremstyle{definition}
\newtheorem{defi}[theo]{Definition}

\theoremstyle{remark}
\newtheorem{rmq}[theo]{Remark}

\theoremstyle{definition}

\theoremstyle{definition}
\newtheorem{ex}[theo]{Examples}

\begin{document}

\title{Non-existing and ill-behaved coequalizers\\of locally ordered spaces}

\date{}

\author{Pierre-Yves Coursolle, and Emmanuel Haucourt\\{\small Laboratoire d'Informatique de l'\'Ecole polytechnique (LIX - UMR 7161),}\\}

\maketitle

\begin{abstract}
Categories of locally ordered spaces are especially well-adapted to the realization of \emph{most} precubical sets \cite{FGR06}, though their colimits are not so easy to determine (in comparison with colimits in the category of \emph{d-spaces} for example \cite[1.4.0]{Grandis09}). 
We use the plural here, as the notion of a locally ordered space vary from an author to another, only differing according to seemingly anodyne technical details.
As we explain in this article, these differences have dramatic consequences on colimits.
In particular, we show that most categories of locally ordered spaces are \emph{not} cocomplete, thus answering a question that was neglected so far. 
The strategy is the following: given a directed loop $\gamma$ on a locally ordered space $X$, we try to identify the image of $\gamma$ with a single point. 
If it were taken in the category of d-spaces, such an identification would be likely to create a vortex \cite[1.4.7]{Grandis09}, while locally ordered space have no vortices.
Concretely, the antisymmetry of local orders gets more points to be  identified than in a mere topological quotient.
However, the effect of this phenomenon is in some sense limited to the neighbourhood of (the image of) $\gamma$.
So the existence and the nature of the corresponding coequalizer strongly depends on the topology around the image of $\gamma$.
As an extreme example, if the latter forms a connected component, the coequalizer exists and its underlying space matches with the topological coequalizer.
\end{abstract}

{\scriptsize \textbf{Keywords}: Directed topology; precubical set; realization; concurrency theory; vortex; colimit of locally ordered spaces;}

\section{Introduction}
The usage of methods from Algebraic Topology in the study of Concurrency Theory was explicitly initiated in \cite{FGR06}. 
One of its key ingredient is the \emph{realization} of \emph{precubical sets}\footnote{Precubical sets are to higher dimensional automata as graphs are to automata. 
For a detailed account of the importance of higher dimensional automata in concurrency theory see \cite{vanGlabbeek06}.
The construction described in \cite{FGR06} actually requires some restrictions on the precubical sets to be realized.} in the category of locally ordered spaces (Definition \ref{defi-lpo}).
The construction of realizations requires certain colimits to exist in the category under consideration, which is guaranteed as soon as it is cocomplete. Unfortunately `the' category of locally ordered spaces is not. We provide several variants of this category, only differing by the separation properties their underlying spaces are required to satisfy. Quite surprisingly, these variations have dramatic consequences on colimits, even on their mere existence. Yet, for the rest of the introduction, we remain vague about the variant under consideration and write $\LPO$ to denote one of them. After having provided a catalog of pathological coequalizer in $\LPO$ (and proven that some of them not even exist), our purpose is to give an insight into locally ordered spaces colimits. 

No algebraic topologist would imagine working in categories that are not cocomplete\footnote{E.g. every model category is required to be bicomplete.}, or in which certain seemingly obvious colimits are so ill-behaved. It is thus necessary, before dwelving in the technical details, to explain why locally ordered spaces should be taken seriously. As a first argument, we cite the work by J. D. Lawson in which the equivalence between \emph{ordered manifolds} (a certain kind of locally ordered spaces) and \emph{conal manifolds} is established \cite[Theorem 2.7]{LawsonJD89}, the relation to Lie theory of semigroups \cite{HHL89}, and also to \emph{causal orientation} in cosmology \cite[pp.22--28]{Segal76}. Beyond that somewhat argument of authority, and adopting the computer scientist point of view, the crucial property of locally ordered spaces is that they are free of \emph{vortices} (a vortex is a point every neighbourhood of which contains a non-trivial directed loop). Regardless of the chosen perspective, vortices are pathological. Then we have to face a dilemma.  
On one hand, we can require our working category to be \emph{topological} over $\Top$ \cite[7.3]{Borceux94b} so colimits be well-behaved. We thus have a convenient framework for homotopical methods, but in which vortices are pervasive.
This is the case with the category of \emph{d-spaces} \cite[1.4.7]{Grandis09}. 
On the other hand, we can ban vortices from our class of models, though this comes at the price of missing or poorly behaved colimits.
This is the case with all the variants of the category of locally ordered spaces considered in this article. 

A natural idea to prove that $\LPO$ is not cocomplete consists of identifying all the points visited by a directed loop to contradict the fact that a locally ordered space has no vortex. We experiment this approach on the standard directed cylinder $\tcircle\times\real$, see Example \ref{exCylinder} \footnote{The latter is, in particular, a conal manifold in the sense of 
\cite{LawsonJD89}: the cones are induces by the canonical parallelization \cite[Appendix 3B]{BG80}.}.
More precisely, one tries to `create' a vortex by identifying all the points of the form $(s,0)$ with $s\in\dcircle$. Depending on the category of locally ordered spaces under consideration, this coequalizer may or may not exist, see Corollaries \ref{corol:colimiteCylindre}, \ref{corol:colimiteCylindreNachbin}, and \ref{corol:colimiteCylindreHausdorff} in the third section of this article.

A lucid analysis of the directed cylinder example reveals that any coequalizer in $\LPO$ (when it exists) is obtained by identifying points of a topological space. Nevertheless, the antisymmetry locally imposed by the elements of an ordered basis (Definition \ref{defi-orderedBasis}) often forces much more points to be identified than in an ordinary topological quotient. Yet, the effects of the phenomenon described above may be limited by the underlying topology. Formally, the more points are ordered, the more points are identified by quotient construction. Dually, the finer  the topology is, the less extra points are identified. The fourth and fifth sections illustrate this claim:  

In the fourth section, we start again from the standard cylinder. However, we equip it with a local order so that it contains a countable family of pairwise disconnected directed loops which converges, in a certain sense, to another directed loop $\gamma$.\footnote{The directed loops mentioned here are to be understood as nontrivial.} Trying to identify all the points visited by $\gamma$ results in a diagram whose coequalizer does not exist. 

In the fifth section, we identify a section of a cylinder whose basis is totally disconnected. In this case, the coequalizer exists, and its underlying topology even matches the coequalizer of the underlying topological spaces.

Despite the pathological behaviour of colimits in the category of locally ordered spaces, a wide class of precubical sets can be realized in it. Moreover, for any precubical set of this class, the underlying space of the realization in $\LPO$ matches the realization in $\Top$ \cite{FGR06}. 
In contrast, we observe that the above property is no longer satisfied if we consider cubical sets instead of precubical ones. Indeed, identifying a section of the directed cylinder results in a colimit that is very close to the one required to realize the following cubical set:  
$$
K_2\quad=\quad\{s\}\ 
%,\qquad K_1\quad=\quad\{c,f\}\ ,\qquad K_0\quad=\quad\{u,d\} 
\qquad\text{with}\qquad 
\partial^+_{1}s=\partial^-_{1}s 
\qquad
\sigma\partial^-_{0}\partial^-_{0}s=\partial^-_{0}s\ . 
$$
Concretely, this cubical set identifies the two vertical edges of the square $s$ and reduces the lower horizontal edge to a single point.

Hence, it seems that locally ordered spaces have been especially tailored for precubical set realization. 
This observation, together with the fact that they naturally occur in some well established branches of mathematics and physics, have convinced the authors that locally ordered spaces deserve a special attention.

\section{Locally ordered spaces}

For all basic definitions related to General Topology, we refer to the standard textbooks \cite{Munkres17} and \cite{Kelley55}. The partial order of a poset $P$ is denoted by $\leq_P$, its underlying set by $P$ (or $|P|$ when we need to emphasize on the distinction).  
\begin{defi}[Ordered basis]\label{defi-orderedBasis}
Let $X$ be a topological space. An ordered basis on $X$ is a set $\overrightarrow{\mathfrak{B}}$ of partially ordered sets such that:

\begin{itemize}
    \item the underlying sets of the elements of $\obase{}$ form a basis of the topology of $X$, and 
%    collection $\{\:|B|\:\:|\:B\in\overrightarrow{\mathfrak{B}}\}$ is a basis of topology which generates $X$, and
    \item for all $x\in X$ and all $B,B'\in\overrightarrow{\mathfrak{B}}$ such that $$x\:\in\: B\cap B'$$ there is $B''\in\overrightarrow{\mathfrak{B}}$ such that $$x\:\in\:B''\:\subset\:B\cap B'$$ and the partial order $\leq_{B''}$ is so that $p\leq_{B''}q$ {implies} $p\leq_{B}q$ and $p\leq_{B'}q$. Since this relation between ordered subsets is pervasively used throughout the paper, we give it a name: for every pair of ordered sets $B$ and $B'$, we denote $B\subset_{lax}B'$ when $|B|\subset |B'|$, and $p\leq_B q$ implies $p\leq_{B'} q$ for all $p,q\in B$.
\end{itemize}
The basis is said to be \emph{strict} when, in the above definition, the partial order $\leq_{B''}$ actually coincides with the restrictions of $\leq_{B}$ and $\leq_{B'}$ to $B''$; this stronger relation will be denoted by $B\subset_{str}B'$. Most of the examples met in this article are of the latter type.

An ordered basis $\overrightarrow{\mathfrak{B}}'$ is coarser than $\overrightarrow{\mathfrak{B}}$ if, for every $x\in X$ and for every $B'\in\overrightarrow{\mathfrak{B}}'$ such that $x\in B'$, there is $B\in\overrightarrow{\mathfrak{B}}$ such that $$x\:\in\: B\subset_{lax} B'\:.$$
We say that $\overrightarrow{\mathfrak{B}}$ and $\overrightarrow{\mathfrak{B}}'$ are equivalent when, in addition, $\overrightarrow{\mathfrak{B}}$ is coarser than $\overrightarrow{\mathfrak{B}}'$.
\end{defi}

It is natural to define a notion of strict equivalence between strict basis by replacing the order $\subset_{lax}$ in the above definition by $\subset_{str}$ but it does not bring anything new:

\begin{prop}
Two strict ordered basis $\obase{}$ and $\obase{'}$ on the topological space $X$ are equivalent if, and only if, they are strictly equivalent.
\end{prop}

\begin{proof}
Two strictly equivalent basis are equivalent because the relation $\subset_{str}$ is stronger than the relation $\subset_{lax}$, .

Conversely, assume that $\obase{}$ and $\obase{'}$ are equivalent.  Let $B\in \obase{'}$ containing a point $x\in X$. 
%Since $\obase{}$ and $\obase{'}$ are equivalent, 
There exist $A\in \obase{}$ such that $x\in A \subset_{lax} B$, and $B'\in \obase{'}$ such that $x\in B' \subset_{lax} A$. Since the basis $\obase{'}$ is strict, there is $B''\in\obase{'}$ such that $x\in B'' \subset_{str} B,\:B'$. Once again, since $\obase{}$ and $\obase{'}$ are equivalent, there is $A'\in \obase{}$ such that $x\in A' \subset_{lax} B''$. Since $\obase{}$ is a strict basis, there is $A''\in\obase{}$ such that $x\in A'' \subset_{str} A,\:A'$. We now check that $A''\subset_{str} B$. Since $A''\subset_{str} A\subset_{lax} B$, we have $A''\subset_{lax} B$. Let $x',x''\in A''$ such that $x'\leq_{B} x''$. We have $x'\leq_{B''} x''$ because $B'' \subset_{str} B$. From $B'' \subset_{str} B'$ and $B' \subset_{lax} A$, we deduce that $x'\leq_{A} x''$. Finally, since $A'' \subset_{str} A$, we obtain $x'\leq_{A''} x''$.
The other direction in the definition of strict equivalence is obtained by symmetry. The above reasoning is summarized in the following diagram:

\[
\begin{tikzcd}
{} & {A} \ar["\emph{lax}",sloped,r] & {B} \\
{} & {B'} \ar["\emph{lax}"',sloped,u]& {} \\
{} & {B''} \ar["\emph{str}"',sloped,u]  \ar["\emph{str}"',sloped,bend right,ruu] & {} \\
{A''} \ar["\emph{srt}",sloped,bend left,ruuu] \ar["\emph{str}"',sloped,r] & {A'}\ar["\emph{lax}"',sloped,u]  & {} \\
\end{tikzcd}
\]
\end{proof}

\begin{defi}\label{defi-prod_ord_bas}
If $\obase{}$ and $\obase{'}$ are ordered basis of topological spaces $X$ and $X'$ then the collection 
$$
\obase{}\times\obase{'}
\quad
=
\quad
\big\{B\times B'\ \big|\ B\in\obase{}\text{ and }B'\in\obase'\big\}
$$
is an ordered basis on $X\times X'$. Note that if $\obase{}$ and $\obase{'}$ are strict then so is $\obase{}\times\obase{'}$.
\end{defi}

The equivalent class of $\mathfrak{B}$ admits a greatest element $\overrightarrow{\mathcal{O}}(\overrightarrow{\mathfrak{B}})$ with respect to inclusion. Its elements are the partially ordered sets $A$ such that: \begin{itemize}
	\item the underlying set of $A$ is included in $X$, and
	\item for all $x\in A$, there exists  $B\in\overrightarrow{\mathfrak{B}}$ such that $x\:\in\: B\subset_{lax} A$.
\end{itemize}

One readily checks that the underlying set of any element of $\overrightarrow{\mathcal{O}}(\obase{})$ is open in $X$.

\begin{lem}
\label{lem:1}
Let $\overrightarrow{\mathfrak{B}}$ be an ordered basis on the topological space $X$, and $O$ be an element of $\overrightarrow{\mathcal{O}} (\obase{})$.
Every open subset $O'$ of $O$ 
equipped with the restriction of $\leq_O$ to $O'$, belongs to $\overrightarrow{\mathcal{O}}(\overrightarrow{\mathfrak{B}})$.
\end{lem}

\begin{proof}
Let $x\in O'$. Since $O'$ is an open of $X$, there exists $B\in \overrightarrow{\mathfrak{B}}$ such that $x\:\in\:B\subset O'$.
Then we have $B'\in \overrightarrow{\mathfrak{B}}$ such that $x\in B'\subset_{lax} O,\:B$.
Thus, we have $B'\subset O'$ and $\leq_{B'}$ is included in $\leq'$.
\end{proof}

\begin{defi}
An \emph{ordered space} is a topological space $X$ equipped with a partial order. A \emph{Nachbin ordered space} is an ordered space whose partial order is closed as a subspace of the product $X\times X$.\footnote{Nachbin ordered spaces should not be confused with Nachbin-Hewitt spaces, which is another name for `realcompact spaces' \cite[p.166]{Johnstone82}.}
\end{defi}

\begin{defi}[Locally ordered spaces]\label{defi-lpo}
A \emph{locally ordered space} is an ordered pair $(X,\Ecal)$ where $\Ecal$ is an equivalence class of ordered basis on the topological space $X$. 
The greatest element of $\Ecal$, whose elements are called the \emph{open posets} of $(X,\Ecal)$, is denoted by $\overrightarrow{\mathcal{O}}(X,\Ecal)$. 
We will often use the same denotation for a locally ordered space and its underlying topological space. 
A locally ordered space is said to be Hausdorff when so is its underlying topological space. A \emph{strictly locally ordered space} is an ordered space $(X,\Ecal)$ such that $\Ecal$ contains a strict ordered basis.
\end{defi}

\begin{defi}
A \emph{locally Nachbin ordered space} $X$ is a locally ordered space such that for every $x\in X$ and every $O\in \overrightarrow{\mathcal{O}}(X)$ containing $x$, there exists $O'\in \overrightarrow{\mathcal{O}}(X)$ containing $x$, which is a Nachbin ordered space (with the topology inherited from $X$)  
such that $O'\subset_{lax}O$.
\end{defi}

In strictly locally ordered spaces, we have the following characterisation :

\begin{prop}
Assume that $(X,\Ecal)$ is a strictly locally ordered space with a chosen strict basis $\obase{}\in\Ecal$. Then $(X,\Ecal)$ is a locally Nachbin ordered space if, and only if, for every $x\in X$, there is $B\in\obase{}$ containing $x$ such that $B$ is a Nachbin space.
\end{prop}

\begin{proof}
Let $x\in X$. There is $B\in \obase{}$ containing $x$. By hypothesis, we have $O\in \overrightarrow{\mathcal{O}}(X)$ such that $x\in O\subset_{lax} B$ and the partial order $\leq_O$ is closed. Then there is $B'\in \obase{}$ such that $x\in B'\subset_{lax} O$. Since $\obase{}$ is a strict ordered basis, there is $B''\in \obase{}$ such that  $x\in B''\subset_{str} B,\: B'$. 
We check that $\leq_{B''}$ is a closed partial order. 
Let $x,x'\in B''$. If $x\leq_{B''} x'$, then $x\leq_{B'} x'$ and $x\leq_{O} x'$. The other way round, if $x\leq_{O} x'$ then $x\leq_{B} x'$, and we also have $x\leq_{B''} x'$ because $B''\subset_{str} B$. We have proven that $\leq_{B''}$ is the restriction of the closed relation $\leq_O$ to $B''$. 
The converse implication is obvious.
\end{proof}

By \cite{Nachbin65}, the underlying topological space of a locally Nachbin ordered space is a locally Hausdorff space.

\begin{rmq}\label{rmq:sublocalpospace}
If $\obase{}$ is an ordered basis of the locally ordered space $X$, and $Y$ is a subspace of the underlying space of $X$, then $\{\ Y\cap B\ |\ B\in\obase{}\ \}$ is an ordered basis on $Y$ (note that if $\obase{}$ is strict then so is this basis). All the ordered basis of $Y$ obtained this way are equivalent, which allows us to define the (strictly) locally ordered subspace $Y$ of $X$. 
\end{rmq}

\begin{rmq}\label{rmq:lpo_prod}
Let $X$ and $X'$ be two locally ordered spaces: an ordered basis of $X\times X'$ is 
given by $\obase{}\times\obase{'}$, with $\obase{}$ and $\obase{'}$ being any ordered basis of $X$ and $X'$ respectively. The equivalence class of $\obase{}\times\obase{'}$ only depends on the equivalence classes of $\obase{}$ and $\obase{'}$.
\end{rmq}

\begin{rmq}\label{rmq:topToLocOrd}
Every topological space can be seen as a strictly locally ordered space with a canonical strict ordered basis which consists of all the open subsets of $X$ equipped with the equality. It is a locally Nachbin ordered space if, and only if, it is a locally Hausdorff space.
\end{rmq}

\begin{defi}\label{def:morphismLPO}
Let $X$ and $Y$ be locally ordered spaces and let $\overrightarrow{\mathfrak{B}}$ (resp. $\overrightarrow{\mathfrak{B}}'$) be an ordered basis in the equivalent class of ordered basis of $X$ (resp. $Y$). A function $f:X\to Y$ is \textit{locally increasing} at $x\in X$ when, for all $B'\in \overrightarrow{\mathfrak{B}}'$ such that $f(x)\in B'$, there exists $B\in\overrightarrow{\mathfrak{B}}$ such that $x\in B$, $f(B)\subset B'$ and $f_{B}:B\to B'$ is increasing. One verifies that this notion only depends on the equivalence classes of $\obase{}$ and $\obase{'}$.
\end{defi}

In the case where the target space is a strictly locally ordered space, we have a convenient characterisation of locally increasing maps.

\begin{prop}\label{prop:chlocallyincreas}
Let $X$ be a locally ordered space, $Y$ be a strictly locally ordered space given by a strict ordered basis $\obase{'}$  and let $f:X\to Y$ be a function. The map $f$ is locally increasing at $x\in X$ if, and only if, it is continuous at $x$ and there exists $O\in\overrightarrow{\mathcal{O}}(X)$ and $B_0\in\obase{'}$ such that $$x\in O, \quad f(O)\subset B_0,\text{ and\quad} f_{O}:O\to B_0\text{ is increasing}$$ where $f_{O}$ is the restriction of $f$ to $O$.
\end{prop}

\begin{proof}
The condition is clearly necessary.
On the other hand, we suppose that there exists $O\in\overrightarrow{\mathcal{O}}(X)$ and $B_0\in\obase{'}$ such that $x\in O$, $f(O)\subset B_0$ and $f_{O}:O\to B_0$ is increasing.
Let $B\in\obase{'}$ such that $f(x)\in B$. There is $B'\in \obase{'}$ such that $f(x)\in B'\subset_{str} B_0,\: B$.
Since $f$ is continuous at $x$, there is an open $O'\subset O$ such that $x\in O'$ and $f(O')\subset B'$.
We denote by $\leq_{O'}$ the restriction of $\leq_O$ to $O'$. By lemma \ref{lem:1} the poset  $(O',\leq_{O'})$ belongs to $\overrightarrow{\mathcal{O}}(X)$.
Then we have $f(O')\subset B'\subset B$, and:
\begin{itemize}
\item[-] the map $f_{O'}: O'\to B_0$ is increasing since $f_{O}:O\to B_0$ is increasing and $\leq_{O'}$ is the restriction of $\leq_O$, 
\item[-] the map $f_{O'}: O'\to B'$ is increasing since $f_{O'}: O'\to B_0$ is increasing, $f(O')\subset B'$, and $\leq_{B'}$ is the restriction of $\leq_{B_0}$, and finally  
\end{itemize}
the map $f_{O'}: O'\to B$ is increasing since $f_{O'}: O'\to B'$ is increasing and the relation $\leq_{B'}$ is the restriction of the relation $\leq_{B}$ to $B'$.
\end{proof}

\section{Cylinder}\label{section:cylinder}

\noindent
The compact unit circle (with its usual topology) is $$S^1\quad:=\quad\{\:z\in\mathbb{C}\quad|\quad\vert z\vert =1\:\}\quad=\quad\{e^{ix}|\:x\in \mathbb{R}\}\quad.$$

\noindent
\begin{defi}\label{def:dcirc}
An \emph{ordered arc} is 
a \textit{proper open arcs} of $S^1$, \ie\ a subset of the form
$$\arcc{ab}\quad:=\quad\{e^{ix}|\:x\in\:]a,b[\}
$$
with $a,b\in\mathbb{R}$ such that $0< b-a < 2\pi$, equipped with the \textit{standard} partial order 
$$e^{ix}\leq_{a,b} e^{iy}\quad\text{ if }\quad a<x<y<b\quad.$$

We observe that if we have $a',b'$ such that $\arcc{ab}=\arcc{a'b'}$, then their standard orders match. So we denote by $\leq_\alpha$ the standard partial order on a proper open arc $\alpha$.
\end{defi}

One readily checks that the proper open arcs with their standard partial order form a strict ordered basis. The resulting locally ordered space is the \textit{directed (unit) circle}, we denote it by $\overrightarrow{S^1}$. The \textit{unordered (unit) circle} is obtained the same way, replacing the standard partial orders on open proper arcs by the discrete ones.\\

The counter-examples we are about to describe are based on  products of locally ordered spaces of the following form (with $X$ denoting any locally ordered space)
$$\dcircle\times X$$

\noindent
We write $p_2:S^1\times X\to X$ for the second projection, and 
$$i_t:S^1\to S^1\times X,\quad s\mapsto (s,t)\qquad(t\in X)$$ for the \textit{section} at the level $t$.\\

    \begin{figure}[!h]
    \caption{Cylinder with $X=[0,1]$}
    \begin{center}
    \hspace{-3cm}
    \includegraphics{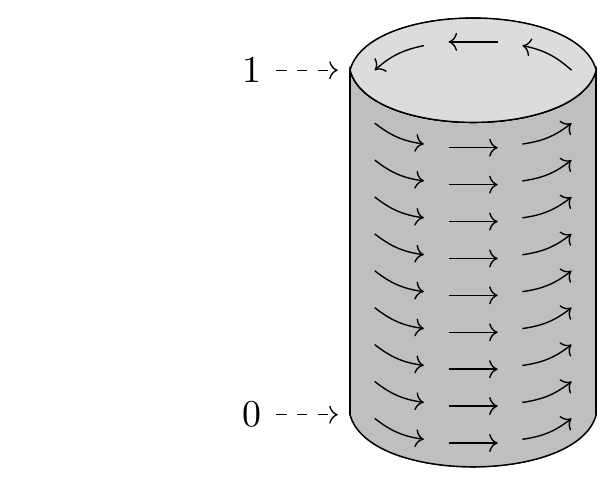}
    \end{center}
    \end{figure}

We fix a point $\ast$ of $X$ and give a criterion on the lattice of neighbourhoods of $\ast$ for the coequalizer of the pair 
$(i_\ast,c_\ast):S^1 \to \overrightarrow{S^1}\times X$ (with $c_\ast:=s\mapsto (1,\ast)$) to exist. 

\noindent
Let $f:\overrightarrow{S^1}\times X\to Y$ be a locally increasing map and $K(f)$ be the set 
\begin{equation}\label{eqn:Kf}
\{t\in X|\: \forall s,s'\in S^1, f(s,t)=f(s',t)\}\ .
\end{equation}

\begin{lem}\label{lemOuvert1}
The set $K(f)$ is an open subset of $X$.
\end{lem}

\begin{proof}
Let $t_0\in K(f)$ and let $U\in\overrightarrow{\mathcal{O}}(Y)$ such that $f(1,t_0)\in U$.

For all $s\in S^1$, $f$ is locally increasing at $(s,t_0)$, so, by the Definition \ref{def:morphismLPO}, there exists 
an ordered arc $\alpha_s$ containing $s$, an open poset $O_s$ of $X$ containing $t_0$ such that $f(\alpha_s\times O_s)\subset U$ and the restriction $f_s:\alpha_s\times O_s\to U$ is increasing.
Since $S^1$ is compact, the open covering made of the proper open arcs of the form $\alpha_s$ admits a finite subcovering  
$$
{\mathcal F}\quad=\quad
\big\{
\alpha_s
\ \big|\ %
s\in J
\big\}\quad.
$$
We denote by $O$ the finite intersection 
$$
\bigcap_{s\in J} O_s\ ,
$$
which is thus an open neighbourhood of $t_0$. 
We are to show that $O\subset K(f)$. Let $t\in O$, $x, y\in S^1$. There exists a finite sequence $r_0,\ldots,r_n\in S^1$ such that 
$r_0=x$, $r_n=y$, and for all $k\in \{1,\ldots,n\}$  
there is $s_k\in J$ such that $r_k,r_{k+1}\in\alpha_{s_k}$ and $r_k$ is less than $r_{k+1}$ in $\alpha_{s_k}$. In particular $(r_k,t)$ is less than $(r_{k+1},t)$ in the product poset $\alpha_{s_k}\times O_{s_k}$  
    from which we deduce that $$f(r_k,t)\leq_{U}f(r_{k+1},t)$$ because the restriction $f_{s_k}$ is increasing. By transitivity of $\leq_U$, we have $f(x,t)\leq_U f(y,t)$. 
  By swapping the roles of $x$ and $y$ in the previous reasoning we prove that $f(y,t)\leq_U f(x,t)$. From the anti-symmetry of $\leq_U$, we deduce that $f(x,t)=f(y,t)$, so $t$ belongs to $K(f)$, which is therefore open in $X$.
\end{proof}
\begin{rmq}
\label{rmq:Kf_open}
Lemma \ref{lemOuvert1} remains valid if one replaces the directed 
circle by a compact locally ordered space that is \emph{strongly 
connected} in the sense that for every ordered pair of points $(a,b)$ there is a directed path from $a$ to $b$.
\end{rmq}
By the above Lemma, if $f\circ i_\ast=f\circ c_\ast$, then $K(f)$ is an open neighbourhood of $\ast$ in $X$. 

We will see that if $f$ collapses the section at level $\ast$ to the point $\ast$, then the collapsing spreads around the sections whose level are close to $\ast$.\\

Now, for every open neighbourhood $O$ of $\ast$, we construct a locally ordered space $X_O$ and a locally increasing map $q_O:\overrightarrow{S^1}\times X\to X_O$ such that $K(q_O)=O$.

We define the set 
$$
X_O\quad:=\quad
O
\quad\sqcup\quad 
\{(s,t)\:\big|\: s\in S_1,\ t\in X/O\}
$$
and the (set theoretic) map 
$q_O : \overrightarrow{S^1}\times X\to X_O$ by 
\begin{equation}
\label{eqn:q_O}
q_O(s,t)=\left\lbrace 
\begin{array}{ll}
t & \text{if } t\in O \\[+2mm] 
(s,t) & \text{if } t\in X/O \ .
\end{array} \right.
\end{equation}

We note that $K(q_O)=O$. 

\begin{lem}
The final topology of $q_O$ is generated by 
the subsets of the form 
$$
U_{\alpha,A}\quad:=\quad
O\cap A
\quad
\sqcup
\quad
\{\ (s,t)\ \big|\:s\in\alpha,\ t\in (X/O)\cap A \}
$$
with 
$\alpha$ proper open arc, and $A\in\overrightarrow{\mathcal{O}}(X)$. 
\end{lem}
\begin{proof}
Indeed, those subsets are open in the final topology because $q^{-1}_O(U_{\alpha,A})=S^1\times (O\cap A)\:\cup\: \alpha\times A$ is an open subset of $S^1\times X$. Conversely, let $B$ be a subset of $X_O$ such that $q^{-1}_O(B)$ is an open subset of $S^1\times X$ and $x\in B$. Since $q_O$ is a surjection, there exists $(s,t)\in S^1\times X$ such that $q_O(s,t)=x$. Then $q^{-1}_O(B)$ is an open neighbourhood of $(s,t)$ in $S^1\times X$ so there are a proper open arc $\alpha$ and an open poset $A$ of $X$ such that $(s,t)\in\alpha\times A\subset q^{-1}_O(B)$. We verify that $U_{\alpha,A}\subset B$.
\end{proof}
An element $u\in U_{\alpha,A}$ is either an element of $O\cap A$ or an ordered pair $(s,t)\in\alpha\times ((X/O)\cap A)$. Depending on the case, the \emph{second component} of $u$ refers to $u$ itself or to $t$, we denote it by $p_2(u)$.
We provide every set $U_{\alpha,A}$ with the image of the partial order $\leq_{\alpha\times A}$ under the mapping $q_O$, which we denote by $\leq^1_{\alpha,A}$.
\begin{lem}
The relation $\leq^1_{\alpha,A}$ matches with the relation $\sqsubseteq$ defined below:
$$
u\sqsubseteq u'
\quad\text{if}\quad
p_2(u)\leq_A p_2(u')
\quad\text{and}\quad
\left\{
\begin{array}{l}
\{u,u'\}\cap O\not=\emptyset\\[+2mm] \text{or} \\[+2mm]
u=(s,t),\ u'=(s',t'),\ \text{and}\ s\leq_\alpha s'.
\end{array}
\right.
$$
\end{lem}
\begin{proof}
Let $u$ and $u'$ be elements of $U_{\alpha,A}$ such that $u\sqsubseteq u'$.

\begin{itemize}
	\item If $u=t\in O$ and $u'=t'\in O$, let $s_0\in\alpha$, we have $(s_0,t)\leq_{\alpha\times A}(s_0,t')$, $q_O(s_0,t)=u$ and $q_O(s_0,t')=u'$, hence $u\leq^1_{\alpha,A}u'$. 
	\item If $u=t\in O$ and $u'=(s',t')\in \alpha\times ((X/O)\cap A)$, we have $(s',t)\leq_{\alpha\times A}(s',t')$, $q_O(s',t)=u$ and $q_O(s',t')=u'$, hence $u\leq^1_{\alpha,A}u'$. 
	\item If $u=(s,t)\in \alpha\times ((X/O)\cap A)$ and $u'=t'\in O$, we have $(s,t)\leq_{\alpha\times A}(s,t')$, $q_O(s,t)=u$ and $q_O(s,t')=u'$, hence $u\leq^1_{\alpha,A}u'$. 
	\item If $u=(s,t)\in \alpha\times ((X/O)\cap A)$ and $u'=(s',t')\in \alpha\times ((X/O)\cap A)$, we have $(s,t)\leq_{\alpha\times A}(s',t')$, $q_O(s,t)=u$ and $q_O(s',t')=u'$, hence $u\leq^1_{\alpha,A}u'$. 
\end{itemize}
The fact that $u\leq^1_{\alpha,A}u'$ implies $u\sqsubseteq u'$ readily derives from the definition of $q_O$.
\end{proof}
\begin{lem}
The transitive closure of the relation $\leq^1_{\alpha,A}$, which we denote by $\leq_{\alpha,A}$, is antisymmetric. Moreover we have 

$u\leq_{\alpha,A}u'$ if and only if

$$ u\sqsubseteq u'\text{\quad or \quad} \exists\: u''\in O\cap A \text{\quad such that \quad} u\sqsubseteq u''\sqsubseteq u'\:.$$
We write $u \trianglelefteq  u'$ when the above condition is satisfied.
\end{lem}
\begin{proof}
We check that the relation $\leq_{\alpha,A}$ is antisymmetric. Indeed, if $u\leq_{\alpha,A}u'\leq_{\alpha,A}u$ then we have $p_2(u)=p_2(u')$. It follows that $u$, $u'$, and also any $u''$ such that $u\leq_{\alpha,A}u''\leq_{\alpha,A}u'$ all belong to $O$ or to its complement. The first case is obvious, in the second one we have $s\leq_\alpha s'\leq_\alpha s$, from which we deduce that $s=s'$. 
The relation $\unlhd$ is indeed an extension of $\leq^1_{\alpha,A}$ because we can take $u''\in\{u,u'\}$ and is clearly included in $\leq_{\alpha,A}$. 
In order to prove that $\unlhd$ is transitive, we first make an observation about the relation $\sqsubseteq$\::
assume that $u_0\sqsubseteq u_1\sqsubseteq u\sqsubseteq  u_2\sqsubseteq u_3$ with $u\in O\cap A$. We have $p_2(u_0)\leq_A p_2(u)\leq_A p_2(u_3)$, and then $u_0\sqsubseteq u \sqsubseteq u_3$. 
We now check that $\unlhd$ is transitive:
assume that $u_0\unlhd u_1\unlhd u_2$. If $u_0=(s_0,t_0)$, $u_1=(s_1,t_1)$ and $u_2=(s_2,t_2)$, and if $u_0\sqsubseteq u_1\sqsubseteq u_2$ then, by transitivity of $\leq_\alpha$ and of $\leq_A$, we have $u_0\sqsubseteq u_2$. Otherwise, we meet one of the following cases:
\begin{enumerate}
\item one of the elements $u_0$, $u_1$, and $u_2$ belongs to $O\cap A$, or 
\item there exists $\tilde u\in O\cap A$ such that $u_0\sqsubseteq \tilde u\sqsubseteq u_1$ or $u_1\sqsubseteq \tilde u\sqsubseteq u_2$.
\end{enumerate}
In any case the observation we made about $\sqsubseteq$ allows us conclude that there exists $u\in O\cap A$ such that $u_0\sqsubseteq u \sqsubseteq u_2$, and therefore $u_0\unlhd u_2$. Finally, the relations $\sqsubseteq$ and $\leq_{\alpha,A}$ match.
\end{proof}

Consequently, if $(X/O)\cap A$ is order-convex in $(A,\leq_A)$ then $\leq_{\alpha,A}=\leq^1_{\alpha,A}$.

\begin{lem}
The family of posets $(U_{\alpha,A},\leq_{\alpha,A})$ is an ordered basis on $X_O$.
\end{lem}

\begin{proof}
Let $u\in X_O$ such that $u\in U_{\alpha_0, A_0}\cap U_{\alpha_1, A_1}$ with $\alpha_0,\alpha_1$ proper open arcs and $A_0,A_1\in\pyccmd{X}$. If $u\in O$, there exists $A_2\in\pyccmd{X}$ such that $A_2\subset O$ and $u\in A_2\subset_{lax} A_0,\:A_1$. Then $U_{\alpha_0,A_2}$ is such that $u\in U_{\alpha_0,A_2}\subset_{lax} U_{\alpha_0,A_0},\: U_{\alpha_1,A_1}$. We note that if $A_2\subset_{str} A_0,\:A_1$, then $U_{\alpha_0,A_2}\subset_{str} U_{\alpha_0,A_0},\: U_{\alpha_1,A_1}$. If $u=(s,t)\in S^1\times (X/O)$, there exists a proper open arc $\alpha_2$ and $A_2\in\pyccmd{X}$ such that $s\in\alpha_2 \subset\alpha_0\cap\alpha_1$, and $t\in A_2\subset_{lax} A_0,\:A_1$. Then $U_{\alpha_2,A_2}$ is such that $u\in U_{\alpha_2,A_2}\subset_{lax} U_{\alpha_0,A_0},\: U_{\alpha_1,A_1}$. We note that if $A_2\subset_{str} A_0,\:A_1$ and if the subsets $(X/O)\cap A_i$ are order-convex in the posets $(A_i,\leq_{A_i})$, then $U_{\alpha_2,A_2}\subset_{str} U_{\alpha_0,A_0},\: U_{\alpha_1,A_1}$. Finally, $X_O$ is a locally ordered space. 
\end{proof}

\begin{rmq}\label{rmq:XO strict condtion}
Moreover, if there is a \emph{strict} ordered basis $\obase{}$ such that for all $B\in\obase{}$, $(X/O)\cap B$ is order-convex in $B$, then $X_O$ is a \emph{strictly} locally ordered space.
In particular, if $X$ a strictly locally ordered space coming from a topological space (\ref{rmq:topToLocOrd}), then the canonical ordered basis satisfies the latter order-convexity condition.
Besides, when $O$ is a clopen subset, any strict basis $\obase{}$ can be turned into a strict basis satisfying the order-convexity condition by keeping only those elements $B\in\obase{}$ such that $B\subset O$ or $B\subset X/O$. 
\end{rmq} 

\begin{lem}
\label{lem:qO_morphism}
The map $q_O$ is locally increasing.
\end{lem}
\begin{proof}
Let $x=(s,t)\in \overrightarrow{S^1}\times X$ and let $U_{\alpha,A}$ such that $q_O(x)\in U_{\alpha,A}$, with $\alpha$ proper open arc, and $A\in\overrightarrow{\mathcal{O}}(X)$. If $q_O(x)\in O$, we can assume that $A\subset O$. Let $\alpha'$ be a proper open arc containing $s$, then $x\in \alpha'\times A\in\pyccmd{\overrightarrow{S^1}\times X}$, $q_O(\alpha'\times A)\subset U_{\alpha,A}$ and the restriction of $q_O$ to $\alpha'\times A$ is increasing from $\alpha'\times A$ to $U_{\alpha,A}$. If $q_O(x)=(s,t)\in S^1\times (X/O)$, then $x\in \alpha\times A\in\pyccmd{\overrightarrow{S^1}\times X}$, $q_O(\alpha\times A)\subset U_{\alpha,A}$ and the restriction of $q_O$ to $\alpha\times A$ is increasing from $\alpha\times A$ to $U_{\alpha,A}$.
\end{proof} 

\begin{lem}\label{lem:qOequaliser}
Any locally increasing map $f:\overrightarrow{S^1}\times X\to Y$ such that $O\subset K(f)$ factorizes through the map $q_O$ in a unique way.
\end{lem}
\begin{proof}
The map $h:X_O\to Y$ soundly defined by $q_O(s,t)=f(s,t)$ is the only one satisfying $f=h\circ q_0$. 
Since the underlying topology on $X_O$ is the final topology associated to $q_O$, the map $h$ is continuous. 
Let $u\in X_O$ and $W\in\pyccmd{Y}$ such that $h(u)\in W$. 
Let $(s,t)\in \overrightarrow{S^1}\times X$ such that $u=q_O(s,t)$. 
Since $f$ is locally increasing, we have a proper open arc $\alpha$ and an open poset $A$ of $\pyccmd{X}$ such that $(s,t)\in \alpha\times A$, $f(\alpha\times A)\subset W$ and the restriction of $f$ to $\alpha\times A$ is increasing from $\alpha\times A$ to $W$. 
Then $u\in U_{\alpha,A}$ and $h(U_{\alpha,A})=h(q_O(\alpha\times A))=f(\alpha\times A)\subset W$. 
It remains to show that the restriction of $h$ to $U_{\alpha,A}$ is an increasing map from $U_{\alpha,A}$ to $W$. 
Let $u'$ and $u''$ be elements of $U_{\alpha,A}$ such that $u'\leq^1_{\alpha,A} u''$. There exist two elements $x'$ and $x''$ of $\alpha\times A$ such that $x'\leq_{\alpha\times A}x''$, $q_O(x')=u'$ and $q_O(x'')=u''$. 
Hence $h(u')=f(x')\leq_W f(x'') = h(u'')$.
\end{proof}

\begin{corol}\label{corol:colimiteCylindre}
In the category of locally ordered spaces, there exists a coequalizer of $i_\ast$ and $c_\ast$ if, and only if, the family of open neighbourhoods of $\ast$ has a smallest element. If $O$ is such a neighbourhood, then $q_O:\overrightarrow{S^1}\times X\to X_O$ is the coequalizer.
\end{corol}

\begin{proof}
Assume that there is a coequalizer $f:\overrightarrow{S^1}\times X\to Y$ of $i_\ast$ and $c_\ast$. Let $O$ be an open neighbourhood of $\ast$. The map $q_O:\overrightarrow{S^1}\times X\to X_O$ coequalizes $i_\ast$ and $c_\ast$ so there is a map $h:Y\to X_0$ such that $q_O=h\circ f$, hence $K(f)\subset K(q_O)=O$.
Moreover we know from lemma \ref{lemOuvert1} that $K(f)$ is an open neighbourhood of $\ast$ in $X$. 

Conversely, let $O$ be the least element among the open neighbourhoods of $\ast$ in $X$. Given $f:\overrightarrow{S^1}\times X\to Y$ that coequalizes $i_\ast$ and $c_\ast$, the subset $K(f)$ is an open neighbourhood of $\ast$ (lemma \ref{lemOuvert1}) so $O\subset K(f)$. Thus, by lemma \ref{lem:qOequaliser}, there exists a unique factorization of $f$ through $q_O$.
\end{proof}

\begin{corol}\label{corol:colimiteCylindreStrict}
Let $X$ be a strictly locally ordered space. Assume that $\mathcal{V}$ is an open neighbourhood basis of $\ast$ satisfying the following property: for every $O\in \mathcal{V}$ there is a strict ordered basis $\obase{}$ of $X$ such that $(X/O)\cap B$ is order-convex in each $B\in\obase{}$. 
Then, in the category of strictly locally ordered spaces, the coequalizer of $i_\ast$ and $c_\ast$ exists if, and only if, 
$\mathcal V$ has a least element.
If $O$ is the least element of \(\mathcal V\), then $q_O:\overrightarrow{S^1}\times X\to X_O$ is the coequalizer of $i_\ast$ and $c_\ast$.
\end{corol}

\begin{proof}
The proof of Corollary \ref{corol:colimiteCylindre} still holds taking Remark \ref{rmq:XO strict condtion} into account.
\end{proof}

\begin{lem}\label{lemFermébis}
For every continuous map $f:S^1\times X\to Y$ with $Y$ locally Hausdorff, the set $K(f)$ is a closed subset of $X$.
\end{lem}

\begin{proof}
Let $t\in \overline{K(f)}$, we are to show that the continuous map $f\circ i_t:S^1\to Y$ is locally constant, which is sufficient to prove that $f\circ i_t$ is constant (\ie\:$t\in K(f)$) because $S^1$ is a connected space. 

Let $s\in S^1$ and let $U\in \mathcal{O}(Y)$ such that $f(s,t)\in U$ and $U$ Hausdorff.

Since $f$ is continuous at $(s,t)$, there exists an open neighbourhood $\alpha\times A$ of $(s, t)$ such that $f(\alpha\times A)\subset U$ with $\alpha$ denoting a proper open arc and $A$ an open of $X$. From $t \in\overline{K(f)}$, we deduce that there is a net $(t_i)_{i\in I}$ of $A\cap K(f)$ that converges to $t$.
Let $s'\in \alpha$. The nets $(s,t_i)_{i\in I}$ and $(s',t_i)_{i\in I}$ converge respectively to $(s,t)$ and $(s',t)$. 
Each $t_i$ belongs to $K(f)$ hence $f(s,t_i)=f(s',t_i)$. The images of the nets $(s,t_i)_{i\in I}$ and $(s',t_i)_{i\in I}$ under $f$ are thus equal and converge, by continuity of $f$, to $f(s,t)$ and $f(s',t)$. We deduce that $f(s,t)=f(s',t)$ because $U$ is Hausdorff.
%
%and since, for all $i\in I$, $t_i\in K(f)$, the net $(f(s,t_i))_{i\in I}$ converges to $f(s',t)$ and $f(s'',t)$ in $U$. Therefore $f(s',t)=f(s'',t)$ since $U$ is Hausdorff.
\end{proof}

\begin{rmq}
Connectedness is the only property of $S^1$ that is really used in the above proof. Under the stronger assumption that $Y$ is Hausdorff, the above lemma is valid for any topological space instead of $S^1$.
\end{rmq}

\begin{prop}\label{prop:XOhausdorff}
If $X$ is a locally ordered space whose underlying topology is Hausdorff, then the two following propositions are equivalent:
\begin{enumerate}
	\item The underlying topology of the locally ordered space $X_O$ is Hausdorff. 
	\item The open subset $O$ of $X$ is closed.
\end{enumerate}
\end{prop}

\begin{proof}
The first point imply the second one by lemma \ref{lemFermébis} (take $f=q_O$). Conversely, assume that $O$ is closed. Let $u,u'\in X_O$ with $u\neq u'$. We have two situations to consider:
\begin{itemize}
	\item There are $s,s'\in S^1$ and $t,t'\in X$ with $t\neq t'$ such that $u=q_O(s,t)$ and $u'=q_O(s',t')$. Since $X$ is Hausdorff, there are disjoint open posets $A$ and $A'$ containing $t$ and $t'$ respectively. For any proper open arc $\alpha$ containing $s$ and $s'$, the subsets $U_{\alpha,A}$ and $U_{\alpha,A'}$ are disjoint open neighbourhood of $u$ and $u'$ respectively.
	\item There are $t\in X$ and $s,s'\in S^1$ with $s\neq s'$ and $u=q_O(s,t)$ and $u'=q_O(s',t)$. Let $\alpha$ and $\alpha'$ be disjoint proper open arcs containing $s$ and $s'$ respectively and let $A$ be open poset containing $t$ and included in $X/O$ (which is an open subset by hypothesis). Then the subsets $U_{\alpha,A}$ and $U_{\alpha',A}$ are disjoint open neighbourhood of $u$ and $u'$ respectively.
\end{itemize}
\end{proof}

\begin{prop}\label{prop:XOclosedOrder}
Assume that $X$ is a locally Nachbin ordered space. The two following propositions are equivalent:
\begin{enumerate}
	\item The space $X_O$ is locally Nachbin ordered.
	\item The open subset $O$ of $X$ is closed.
\end{enumerate}
\end{prop}

\begin{proof}
The first proposition imply the second one by lemma \ref{lemFermébis}. Conversely, assume that $O$ is closed.
Let $A$, $\alpha$ and $u$ be an element of $\pyccmd{X}$, an ordered arc, and an element of $U_{\alpha,A}$ respectively. 
We have two cases to deal with. 
On one hand, if $u=t\in O\cap A$, there is $A'\in\pyccmd{X}$ included in $O$ such that $t\in A'\subset_{lax} A$, and the partial order $\leq_{A'}$ is closed for the topology induced by $X$. So the sets $U_{\alpha,A'}$ and $A'$ are equal, the topologies induced on them by $X_O$ and $X$ are the same, and the partial orders $\leq_{\alpha,A'}$ and $\leq_{A'}$ coincide. Moreover, we have $U_{\alpha,A'}\subset_{lax}U_{\alpha,A}$. On the other hand, if $x=(s,t)\in \alpha\times (A\cap (X/O))$ with $X/O$ open, there is $A'\in\pyccmd{X}$ included in $X/O$ such that $t\in A'\subset_{lax} A$ and the partial order $\leq_{A'}$ is closed for the topology induced by $X$. So the sets $U_{\alpha,A'}$ and $\alpha\times A'$ are equal, the topologies induced on them by $X_O$ and $S^1\times X$ are the same, and the partial orders $\leq_{\alpha,A'}$ and $\leq_\alpha\!\times\!\leq_{A'}$ coincide. Moreover, we have $U_{\alpha,A'}\subset_{lax}U_{\alpha,A}$. 
\end{proof}

\begin{corol}\label{corol:colimiteCylindreNachbin}
In the category of (strictly) locally Nachbin ordered spaces, there exists a coequalizer of $i_\ast$ and $c_\ast$ if, and only if, the family of clopen neighbourhoods of $\ast$ has a smallest element.
If $O$ is such a neighbourhood, then $q_O:S^1\times X\to X_O$ is the coequalizer.
\end{corol}

\begin{proof}
Assume that there is a coequalizer $f:S^1\times X\to Y$ of $i_\ast$ and $c_\ast$ with $Y$ a (strictly) locally Nachbin ordered space. Then, by lemma \ref{lemOuvert1} and lemma \ref{lemFermébis}, $K(f)$ is a clopen neighbourhood of $\ast$ in $X$. Let $O$ be a clopen neighbourhood of $\ast$. The map $q_O:S^1\times X\to X_O$ coequalizes $i_\ast$ and $c_\ast$ and, by remark \ref{rmq:XO strict condtion} and by proposition \ref{prop:XOclosedOrder}, $X_O$ is a (strictly) locally Nachbin ordered space, so there is a map $h:Y\to X_O$ such that $q_O=h\circ f$, hence $$K(f)\subset K(q_O)=O$$

Conversely, if $O$ is the least element among the clopen neighbourhoods of $\ast$ in $X$. Still by remark \ref{rmq:XO strict condtion} and by proposition \ref{prop:XOclosedOrder}, the space $X_O$ is a (strictly) locally Nachbin ordered space. For every $f:S^1\times X\to Y$ that coequalizes $i_\ast$ and $c_\ast$ with $Y$ (strictly) locally Nachbin ordered, the subset $K(f)$ is an clopen neighbourhood of $\ast$ (see lemmas \ref{lemOuvert1} and \ref{lemFermébis}). Therefore $O$ is included in $K(f)$, and there exists a unique factorization of $f$ through $q_O$.
\end{proof}

\begin{corol}\label{corol:colimiteCylindreHausdorff}
Assume that the underlying topology of $X$ is Hausdorff. In the category of (strictly) locally (Nachbin) ordered Hausdorff spaces, there exists a coequalizer of $i_\ast$ and $c_\ast$ if, and only if, the family of clopen neighbourhoods of $\ast$ has a smallest element $O$. In that case, $q_O:S^1\times X\to X_O$ is the coequalizer.
\end{corol}

\begin{proof}
Assume that there is a coequalizer $f:S^1\times X\to Y$ of $i_\ast$ and $c_\ast$ with $Y$ a (strictly) locally (Nachbin) ordered Hausdorff space. Then, by lemmas \ref{lemOuvert1} and \ref{lemFermébis}, $K(f)$ is a clopen neighbourhood of $\ast$ in $X$. Let $O$ be a clopen neighbourhood of $\ast$. The map $q_O:S^1\times X\to X_O$ coequalizes $i_\ast$ and $c_\ast$ and, by remark \ref{rmq:XO strict condtion} and by propositions \ref{prop:XOhausdorff} and \ref{prop:XOclosedOrder}, $X_O$ is a (strictly) locally (Nachbin) ordered Hausdorff space, so there is a map $h:Y\to X_0$ such that $q_O=h\circ f$, hence $$K(f)\subset K(q_O)=O$$

Conversely, let $O$ be the least clopen neighbourhood of $\ast$. Still by remark \ref{rmq:XO strict condtion} and by propositions \ref{prop:XOhausdorff} and \ref{prop:XOclosedOrder}, the space $X_O$ is a (strictly) locally (Nachbin) ordered space. For every $f:S^1\times X\to Y$ that coequalizes $i_\ast$ and $c_\ast$ with $Y$ a (strictly) locally (Nachbin) ordered space, the subset $K(f)$ is an clopen neighbourhood of $\ast$ (see lemmas \ref{lemOuvert1} and \ref{lemFermébis}) so $O$ is included in $K(f)$. So there exists a unique factorization of $f$ through $q_O$.
\end{proof}

It is now time to provide some examples:

\begin{ex}\label{exCylinder}
The spaces $\real$ and $\rational$ are the real line and 
the space of rational numbers. 
The coarsest refinement of the 
topology of $\real$ in which every singleton $\{x\}$ with $x\not=0$ is open induces a topological space that is denoted by $\real_\star$. 

The collection of open subsets of $\real$, each equipped with the standard order, forms a strictly locally Nachbin ordered space which we denote by $\dreal$. 
Let $S$ be an infinite set with a distinguished element $\bar s$. We denote by $\mathcal U$ the topological space on $S$ in which a subset is open when it contains $\bar s$ \footnote{The open subsets of $\mathcal U$ are the elements of the \emph{principal ultrafilter} on $S$ generated by $\bar s$ \cite[p.233]{DP02}}.
The table here below summarizes the cases where the coequalizer of $i_\ast$ and $c_\ast$ exists (see \ref{rmq:topToLocOrd}). 
The distinguished elements of $\real_\star$ and $\mathcal{U}$ are respectively $0$ and $\bar s$.  
In all the other cases, the distinguished element can be any point $\ast$ of the space.
\begin{center}
\begin{tabular}{|l|c|c|c|c|c|c|}
\hline
%\diagbox[innerleftsep=1.5cm,innerrightsep=1pt]{\scriptsize category of}{$\overset{\phantom{\text{\LARGE A}}}{X\:\:}$} 
\hspace{-2.1mm}\begin{minipage}{6mm}
\begin{tikzpicture}
\draw (0,0.9)--(6.5,0);
\draw (6.3,0.49) node {\(X\)};
\draw (0.6,0.25) node {\scriptsize category};
\end{tikzpicture}
\end{minipage}
& $\real$ & $\dreal$ & $\dcircle$ & $\real_\star$ & $\rational$ & $\mathcal U$ \\
\hline
\scriptsize (strictly) locally ordered spaces & \multicolumn{5}{c|}{\xmark} & \cmark \\\hline 
\scriptsize (strictly) locally ordered Hausdorff spaces & \multicolumn{3}{c|}{\cmark} & \multicolumn{2}{c|}{\xmark} & \cmark \\\hline 
\scriptsize (strictly) locally Nachbin ordered spaces & \multicolumn{3}{c|}{\cmark} & \multicolumn{2}{c|}{\xmark} & n/a \\\hline 
\scriptsize (strictly) locally Nachbin ordered Hausdorff spaces & \multicolumn{3}{c|}{\cmark} & \multicolumn{2}{c|}{\xmark} & n/a \\\hline
\end{tabular}
\end{center}

In the first three columns, the coequalizer, when it exists, 
is the second projection. It is also the coequalizer when the space under consideration is $\mathcal U$ and the ambient category is that of locally ordered \underline{Hausdorff} spaces. In the latter case, if we drop the Hausdorffness assumption, the coequalizer is the quotient map $q_O$ where $O=\{\bar s\}$, see (Eq. \ref{eqn:q_O}).
All the spaces appearing in the above table are strict and satisfy the hypotheses of Corollary \ref{corol:colimiteCylindreStrict}. Consequently, the results summarized in the table are valid regardless of the fact that the coequalizers are taken in categories of strict or lax local orders.
\end{ex}

\section{Zebra cylinder}
We emphasize that any point of $[0,1]$ admits the whole space as its smallest clopen neighbourhood, so the according to the results from Section \ref{section:cylinder} the coequalizer of $(i_0,c_0):S^1\to S^1\times [0,1]$ with the \emph{product ordered basis} on $S^1\times [0,1]$ exists in the category of locally ordered Hausdorff spaces (it is actually $[0,1]$). 

In this Section, we describe a strict ordered basis $\obase{}$ on $S^1\times [0,1]$ so that the coequalizer of the morphisms $(i_0,c_0)$ no longer exists in the category of (strictly) locally (Nachbin) ordered (Hausdorff) spaces. 
As before, the strategy consists of setting the ordered basis in a way that:    
\begin{itemize}
\item all the sections in a chosen neighbourhood $V$ of the section $i_0$ are collapsed, and 
\item the neighbourhood $V$ can be made arbitrarily small.
\end{itemize}

Let $\natbar$ be the set of extended natural numbers $\nat\sqcup \{+\infty\}$, and $(d_n)_{n\in\mathbb N}$ be a strictly decreasing sequence with values in the compact unit interval $I:=[0,1]$. Assume that $\inf\{d_n|n\in\nat\}=0$. For all $n\in\nat$, $I(n)$ denotes the interval $[d_{2n+1},d_{2n}]$ and $I(+\infty)$ the degenerated interval $\{0\}$.

The elements of $\obase{}$ are of the form $\alpha\times O$ where $\alpha$ is an ordered arc and $O$ is an open subset of $I$, partially ordered  as follows:

\begin{equation}\label{eqn:ordZ}
(s,u)\preceq^O_{\alpha}(s',u')\text{\quad if \quad}
\left\{\quad
\begin{gathered}
 s\leq_{\alpha} s'\text{ and } u=u'\in I(n)\ \text{for some $n\in\natbar$}, \\
 \text{or}\ \ s= s'\text{ and } u=u'\ .\hfill
\end{gathered} 
\right.
\end{equation}

We denote by $Z$ the resulting \emph{strictly locally Nachbin ordered Hausdorff space}.
   
    \begin{figure}[!h]
    \caption{Zebra cylinder}
    \begin{center}
    \includegraphics{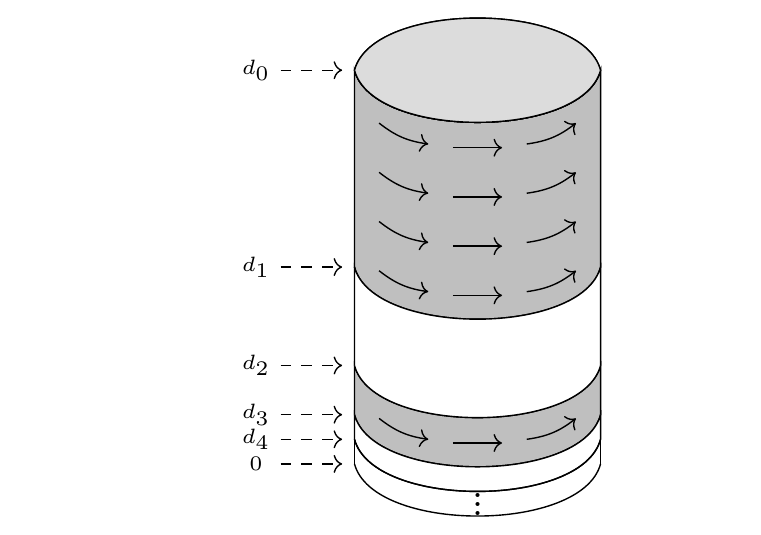}
    \end{center}
    \end{figure}

Let $f:Z\to X$ be a locally increasing map such that $$f\circ i_0=f\circ c_0$$
and recall that $K(f)$ is the set 
\begin{equation*}
\{t\in I|\: \forall s,s'\in S^1, f(s,t)=f(s',t)\}\ .
\end{equation*}
We prove a result similar to lemma \ref{lemOuvert1}.
\begin{lem}\label{lem:ouvertIm}
There exists $n\in\mathbb{N}$ such that, for all $m\geq n$, $I(m)\subset K(f)$.
\end{lem}

\begin{proof}
The arguments are mostly the same as in the proof of Lemma  \ref{lemOuvert1}. 
By hypothesis $0$ belongs to $K(f)$. Let $U\in\overrightarrow{\mathcal{O}}(X)$ such that $f(1,0)\in U$.
For all $s\in S^1$, $f$ is locally increasing at $(s,0)$, so there exists an open neighbourhood $\alpha_s\times O_s$ of $(s, t_0)$ such that $f(\alpha_s\times O_s)\subset U$ and the restriction $f_s:(\alpha_s\times O_s,\preceq_s)\to U$ is increasing with 
$\alpha_s$ denoting a proper open arc and $\preceq_s$ the order $\preceq^{O_s}_{\alpha_s}$ on $\alpha_s\times O_s$.

Since $S^1$ is compact, the open covering made of the proper open arcs of the form $\alpha_s$ admits a finite subcovering  
$$
{\mathcal F}\quad=\quad
\big\{
\alpha_s
\ \big|\ %
s\in J
\big\}\quad.
$$
We denote by $O$ the finite intersection 
$$
\bigcap_{s\in J} O_s\ ,
$$
which is thus an open neighbourhood of $0$. 

Since the sequence $(d_n)_{n\in\nat}$ tends to $0$, there exists 
$n\in\nat$ such that the intervals $I(m)$ are included in $O$ for every $m\geq n$. Let $m$ be such a natural number.
We are to show that $I(m)\subset K(f)$. Given $t\in I(m)$, $x, y\in S^1$, there exists a finite sequence $r_0,\ldots,r_l\in S^1$ such that $r_0=x$, $r_l=y$, and for all $k\in \{1,\ldots,l\}$ there is $s_k\in J$ such that $r_k,r_{k+1}\in\alpha_{s_k}$ and $r_k\leq_k r_{k+1}$ with $\leq_k$ denoting the standard order on the proper open arc $\alpha_{s_k}$. In particular, since $t\in I(m)$, we have $$(r_k,t)\preceq_{s_k}(r_{k+1},t)$$ from which we deduce that $$f(r_k,t)\leq_{U}f(r_{k+1},t)$$ because the restriction $f_{s_k}$ is increasing. By transitivity of $\leq_U$, we have $f(x,t)\leq_U f(y,t)$. 
  By swapping the roles of $x$ and $y$ in the previous reasoning we prove that $f(y,t)\leq_U f(x,t)$. From the anti-symmetry of $\leq_U$, we deduce that $f(x,t)=f(y,t)$, so $t$ belongs to $K(f)$, which therefore contains $I(m)$.
\end{proof}

\begin{prop}
The pair of morphisms $(i_0,c_0:S^1\to Z)$ does not have any coequalizer in the category of (strictly) locally (Nachbin) ordered (Hausdorff) spaces.
\end{prop}

\begin{proof}
Let $f:Z\to X$, $g:Z\to Y$ be two locally increasing maps which coequalize $i_0$ and $c_0$. If there is $h:Y\to X$ such that $f=h\circ g$, then $K(g)\subset K(f)$.

We construct, for each $n\in\mathbb{N}$, a locally ordered space $X_n$ (which is actually strict, Nachbin, and Hausdorff) and a locally increasing map $f_n:Z\to X_n$ such that $f_n\circ i_0=f\circ c_0$ and $K(f_n)=[0,d_{2n}]$. If $g$ was the  coequalizer of $(i_0,c_0)$, we would have 
$$
0\quad\in\quad K(g)\quad \subset\quad \bigcap_{n\in \mathbb{N}} K(f_n)\quad=\quad\bigcap_{n\in \mathbb{N}} [0,d_{2n}] 
\quad=\quad \{0\}
$$ 
but this is in contradiction with lemma \ref{lem:ouvertIm}.
    \begin{figure}[!h]
    \caption{The locally ordered space $X_n$}
    \begin{center}
    \includegraphics{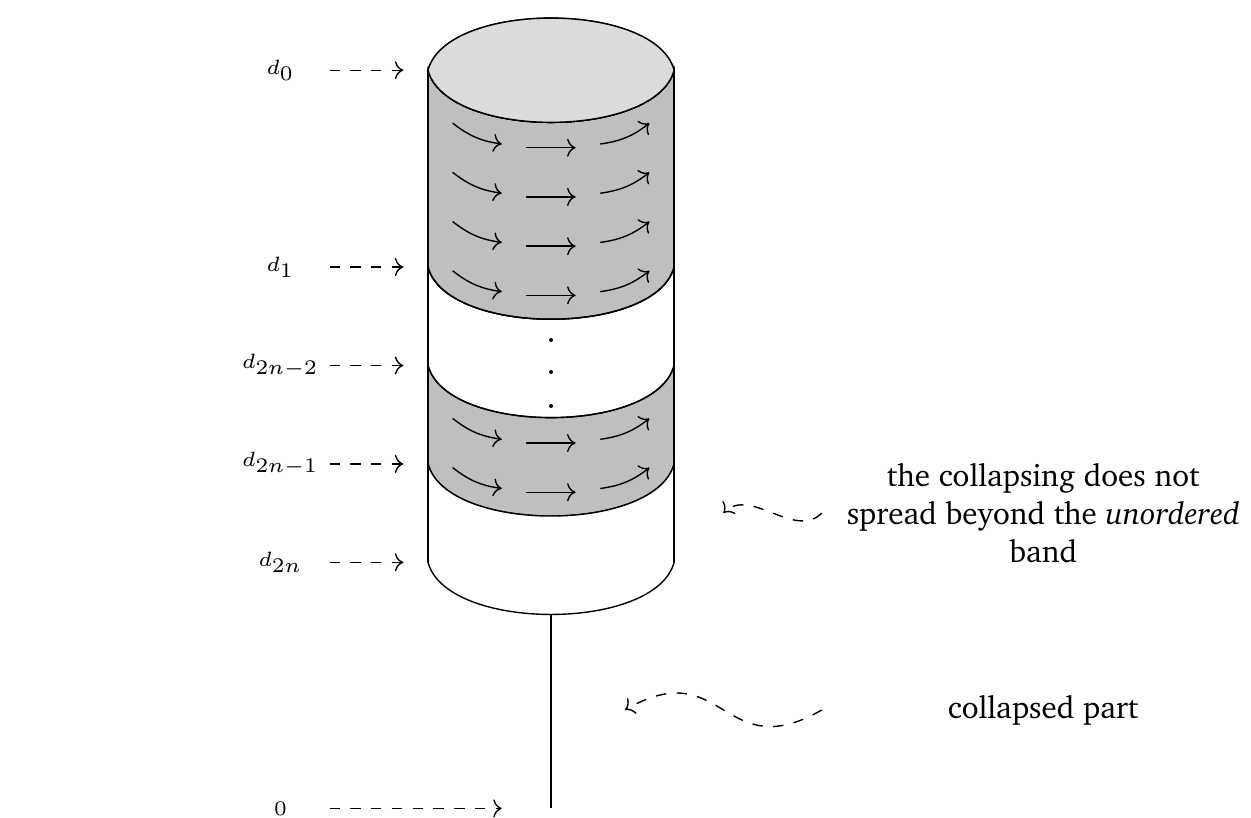}
    \end{center}
    \end{figure}
\noindent
For $n\in\mathbb{N}$ we define the set 
$$
X_n\quad:=\quad
\big[0,\scalebox{1.0}{$d_{2n}]$} 
\sqcup
\big\{(s,t)\:\big|\: s\in S_1, t\in ]\scalebox{1.0}{$d_{2n}$},1]\big\}
$$
and the (set theoretic) map 
$f_n : Z\to X_n$ by $$f_n(s,t)=\left\lbrace 
\begin{array}{ll} 
(s,t) & \text{if } t> d_{2n} \\ 
t & \text{if } t\leq d_{2n}\ .
\end{array} \right.$$
We note that $f_n\circ i_0=f_n\circ c_0$ and $K(f_n)=[0,d_{2n}]$. 
The final topology of $f_n$ is generated by 
the subsets of the form 
$$
O_{\alpha,A}\quad:=\quad
\left\{
\begin{array}{ll}
B & \text{if $A\subseteq\big[0,\scalebox{1.0}{$d_{2n}$}\big[$} \\[+2mm]
B
  \sqcup 
\big(S_1\times C\big)  & \text{if $\scalebox{1.0}{$d_{2n}$}\in A$}\\[+2mm]
\alpha\times C & \text{if $A\subseteq\ \big]\scalebox{1.0}{$d_{2n}$},1\big]$}
\end{array}
\right.
$$
with 
$\alpha$ proper open arc, $A$ an open interval of $I$, $B:=A\cap\big[0,\scalebox{1.0}{$d_{2n}$}\big]$ and $C:=A\cap \big]\scalebox{1.0}{$d_{2n}$},1\big]$. 
In order to define a (strict) ordered basis on $X_n$ that makes the map $f_n$ a morphism of locally ordered spaces, we only consider the intervals $A$ whose length is (strictly) 
less that $d_{2n-1}-d_{2n}$.
The partial order $\leq_{\alpha,A}$ on $O_{\alpha,A}$ is then  defined as the equality in the first and second cases, and matches the partial order described at (\ref{eqn:ordZ}) in the third case.
In particular, the condition on the length of $A$ guarantees that the partial 
order on $B\sqcup (S^1\times C)$ on one hand, and the 
partial order on $\alpha\times A$ from the ordered basis of $Z$ on the other hand (for any proper open arc $\alpha$), both match on $\alpha\times C$: this key observation ensures that $f_n$ is indeed a 
morphism of locally ordered spaces.
\end{proof}

\section{Rational based cylinder}
Previously, we saw that coequalizers in the category of locally ordered spaces may behave differently than in the category of topological space because of the collapsing spreading described in  lemmas \ref{lemOuvert1} and \ref{lem:ouvertIm}. These latter rely on the fact that the directed loops are continuous (\ie they are indexed by $\dcircle$). In this section we replace $\dcircle$ by some of its dense totally disconnected subspace. Then we exhibit a pair of morphisms whose coequalizer exists, and whose underlying space matches with the topological coequalizer. 

\begin{defi}\label{defi-rational_circle}
The subspace $\{\ e^{ix}\ |\ x\in\rational\ \}$ of $S^1$ is denoted by $\Qcircle$.
The \emph{directed rational unit circle} $\dQcircle$ is the subspace $\Qcircle$ with the locally ordered space structure inherited from the directed unit circle, see Remark \ref{rmq:sublocalpospace}.
\end{defi}

We overload the denotations $i_0$ and $c_0$ which now designate 
the mappings $$s\in\Qcircle\mapsto(s,0)\in\dQcircle\times I
\quad
\text{and}
\quad 
s\in\Qcircle\mapsto(1,0)\in\dQcircle\times I$$
As before we identify all the points of the section $\Qcircle\times\{0\}$. We now describe the resulting coequalizer 
in the category of topological spaces. The underlying set is the disjoint union
$$
\{0\}\quad\sqcup\quad
\Qcircle\times]0,1]\quad,
$$
and the quotient map is denoted by $$q\ :\ \Qcircle\times[0,1]\to \{0\}\sqcup(\Qcircle\times]0,1])\quad.$$
A basis of open neighbourhoods of $(s,t)$ with $t>0$ is given by the traces of the products $\alpha\times]a,b[$ with $s\in\alpha$ proper open arc and $0<a<t<b$. 
The associated partial order is given by the restriction of the product order $\leq_\alpha \times =$.

The neighbourhoods of $0$ are a bit harder to describe.
We provide a basis of open neighbourhoods whose elements will be the supports of the partial orders around $0$. 
To this aim, we consider the set $\mathcal{H}$ of all functions $h:\Qcircle\to[0,1]$ which are continuous, strictly positive, and such that
$$
\inf\ h\quad=\quad 0\quad .
$$    
For every function $h\in\mathcal{H}$, we define the set 
$$
O_h
\quad
=
\quad
\{0\}
\quad
\sqcup
\quad
\big\{\ (s,t)\ \in\ \Qcircle\times]0,1]\ \ \big|\ \ t<h(s)\ \big\}\quad.
$$ 
We note that $\mathcal{H}$ is a inf-semilattice\footnote{Any pair of elements has a greatest lower bound} with the minimum being computed pointwise. Moreover, the map $h\mapsto O_h$ is a morphism of inf-semilattices.
We are going to prove that the set of all $O_h$ is a basis of open neighbourhoods of $0$.

Firstly, in order to prove that $O_h$ is an open subset, we  show that $q^{-1}(O_h)=\{(s,t)\in\Qcircle\times I\ |\ t<h(s)\}$ is an open subset of $\Qcircle\times I$. Let $(s,t)\in q^{-1}(O_h)$, since $t< h(s)$, there exists two disjoint open intervals $A$ and $A'$ of $I$ such that $A<A'$\:\footnote{i.e. $a<a'$ holds for all $a\in A$ and $a'\in A'$.}, $t\in A$ and $h(s)\in A'$. By continuity of $h$, there is an open subset $O$ of $\Qcircle$ containing $s$ such that $h(O)\subset A'$. Therefore, $O\times A$ is an open neighbourhood of $(s,t)$ included in $q^{-1}(O_h)$.

Secondly, in order to prove that any open neighbourhood of $0$ contains some $O_h$, we use the following lemma:

\begin{lem}\label{lem:toto}
Let $O$ be an open neighbourhood of $i_0(\Qcircle)$. Let $(s_n)_{n\in\nat}$ be an enumeration of $\Qcircle$. We inductively define a family $(A_j,t_j)_{j\in J}$ (with $J\subseteq\nat$) such that 
\begin{itemize}
\item for all $j\in J$ :
	\begin{itemize}
	\item the set $A_j$ is open,
	\item the real number $t_j$ belongs to $]0,1]$, and
	\item the product $A_j\times [0,t_j[$ is included in $O$,
	\end{itemize}	
\item the sets $A_j$ form a partition of $\Qcircle$, and
\item the greatest lower bound of the set $\{\:t_j\:|\:j\in J\:\}$ is $0$.
\end{itemize}
\end{lem}
Then, for any open neighbourhood $U$ of $0$, the set $q^{-1}(U)$ is an open neighbourhood of $i_0(\Qcircle)$. Therefore, by applying the lemma, we get a family $(A_j,t_j)$ from which we define the map $h:\Qcircle\to[0,1]$ which sends $s\in A_j$ to $t_j$. We observe that it belongs to $\mathcal{H}$, and that $q^{-1}(O_h)\subset q^{-1}(U)$. Consequently, we get $O_h\subset U$. This concludes the proof that the set of all $O_h$ is a basis of open neighbourhoods of $0$.
\begin{proof}[Proof of Lemma \ref{lem:toto}]
Since $O$ is open and contains $(s_0,0)$, there exists an open neighbourhood of $(s_0,0)$ of the form $A_0\times [0,t_0[$. 
Taking $A_0$ to be the trace of a proper open arc whose extremities 
are $e^{ia}$ and $e^{ib}$ with $a$ and $b$ in $\real/\rational$, we 
obtain a clopen subset of $\Qcircle$. Moreover, we choose $A_0$ so that $A_0\not=\Qcircle$. Define $J_0=\{0\}$.
\\\ \\
Suppose that we have already defined $A_j$ and $t_j$ for $j\in J_N$, 
with $N+1$ denoting the cardinality of $J_N$. We actually suppose 
that the following stronger hypotheses are satisfied:
\begin{itemize}
\item each $A_j$ is a clopen, we have $t_j < \frac1{j+1}$, and
\item the family of sets $A_j$, with $j\in J_N$, does not cover $\Qcircle$ though it contains $\{s_0,\ldots,s_N\}$. 
\end{itemize}
Let $n$ be the smallest integer such that $s_n$ does not belong to the union $U_N$ of the sets $A_j$ for $j\in J_N$. We have $n>N$ and 
we define $J_{N+1}=J_N\cup\{n\}$. 
We can find a clopen $A_n$ which contains $s_n$ and 
a number $t_n<\frac1{n+1}$ so that $A_n\times[0,t_n[$ is included in $O$. 
The union $U_N$ is closed because so is each $A_j$, so 
we can suppose that $A_n$ does not meet $U_N$. Of course we can also 
restrict $A_n$ so that $U_N\cup A_n\not=\Qcircle$.
\end{proof}

We equip the sets $O_h$ with partial orders $\leq_h$ so that they become the elements of the expected ordered basis containing $0$.  
\\\ \\
By definition, we have $(s,t)\leq_h(s',t')$ when $t=t'$ and there exists a proper open arc $\alpha$ such that $s\leq_\alpha s'$ and $(\alpha\cap \Qcircle)\times \{t\}\subset O_h$ (and of course $0\leq_h0$).

Let $h,h'\in\mathcal{H}$. Since $O_{\min(h,h')}$ is the intersection of $O_h$ and $O_{h'}$, the partial order $\leq_{\min(h,h')}$ matches the restrictions of both $\leq_{h}$ and $\leq_{h'}$. 

The collection of partially ordered sets $\alpha\times]a,b[$ (with $s\in\alpha$ proper open arc and $0<a<b$) and $O_h$ (with $h\in\mathcal{H}$) thus forms a (strict) ordered basis. We denote by $W$ the corresponding locally order space on $\{0\}\sqcup
\Qcircle\times]0,1]$.

\begin{prop}
The quotient map $q$ induces the coequalizer of $i_0$ and $c_0$. 
\end{prop}
\begin{proof}
One easily checks that the map $q$ is locally increasing. Let $f:\dQcircle\times I\to X$ be a locally increasing map such that $f\circ i_0=f\circ c_0$. 
The underlying topology of $W$ is the final one so we have a unique continuous map $g$ from the underlying space of $W$ to that of $X$ such that $f=g\circ q$. The only point of $W$ around which $g$ is not trivially increasing is $0$. This latter case has to be treated carefully. 
Let $U\in\overrightarrow{\mathcal{O}}(X)$ such that $g(0)\in U$.
Let $(s_n)_{n\in\nat}$ be an enumeration of $\Qcircle$. We construct by induction a family $(\alpha_j,t_j)_{j\in J}$ (with $J\subseteq\nat$) such that 
\begin{itemize}
\item for all $j\in J$ :
	\begin{itemize}
	\item $\alpha_j$ is a proper open arc of the form $\arcc{a_j b_j}$ with $a_j,b_j\in\real/\rational$.
	\item the number $t_j$ belongs to $]0,1]$,
	\item $f((\alpha\cap \Qcircle)\times [0,t_j[)\subset U$, and 
	\item one has 	$f(s,t) \leq_U f(s',t)$ when $s\leq_{\alpha_j}s'$ for $s,s'\in\alpha_j\cap\Qcircle$ and $t\in[0,t_j[$
	\end{itemize}	
\item the sets $\alpha_j \cap \Qcircle$ form a partition of $\Qcircle$.
\end{itemize}
Since $f$ is locally increasing at $(s_0,0)$, there exists an open neighbourhood of $(s_0,0)$ of the form $A_0\times [0,t_0[$ such that the restriction of $f$ to $A_0\times [0,t_0[$ with values in $U$ is increasing. 
Taking $A_0$ to be the trace of a proper open arc $\alpha_0=\arcc{a_0b_0}$ with $a_0$ and $b_0$ in $\real/\rational$, we obtain a clopen subset of $\Qcircle$. Define $J_0=\{0\}$.
Suppose that we have already defined $\alpha_j$ and $t_j$ for $j\in J_N$, 
with $N+1$ denoting the cardinality of $J_N$. 
If the union $U_N$ of the sets $\alpha_j\cap \Qcircle$ for $j\in J_N$ is $\Qcircle$, then $J:=J_n$ and the construction is over. Otherwise, let $n$ be the smallest integer such that $s_n$ does not belong to $U_N$. We have $n>N$ and 
we define $J_{N+1}=J_N\cup\{n\}$. 
We can find a clopen $A_n$ which contains $s_n$ and 
a number $0<t_n\leq 1$ so that the restriction of $f$ to $A_n\times[0,t_n[$ with values in $U$ is increasing. 

The union $U_N$ is closed because so is each $\alpha_j\cap \Qcircle$, so 
we can suppose that $A_n$ does not meet $U_N$.
Finally, we can assume that $A_n$ is the trace of a proper open arc $\alpha_n=\arcc{a_nb_n}$ with $a_n$ and $b_n$ in $\real/\rational$.

For each $j\in J$, consider a continuous map $\phi_j:\alpha_j\to\ [0,t_j]$ that is strictly 
positive, and tends to $0$ on $a_j$ and $b_j$ (basically a \emph{bump function} would be more than enough). Then let $h$ be the map whose the restriction to $\alpha_j$ is 
$$
x\ \mapsto\ \phi_j(x)\qquad\text{with }\quad x\ \in\ \alpha_j
$$
One readily deduces from the definition of $h$ that it belongs 
to $\mathcal H$ and satisfies $g(O_h)\subset U$. 
We now check that the restriction of $g:W\to X$ to $O_h$ 
is order-preserving from $\leq_h$ to $\leq_U$. 
Suppose that we have $(s,t)\leq_h(s',t)$ in $O_h$.
By definition of the partial order $\leq_h$ there exists 
an open proper arc $\alpha$ such that $s\leq_\alpha s'$ and 
$(\alpha\cap\Qcircle)\times\{t\}\subseteq O_h$.
Let $j\in J$ be such that $s\in\alpha_j$. Recall that $\alpha$ and $\alpha_j$ are the images of $]a,b[$ and $]a_j,b_j[$ under the complex exponential map $t\in\real\mapsto e^{it}\in 
S^1$ with $b-a<2\pi$, $b_j-a_j<2\pi$. Moreover, once $a$ and $b$ are fixed, one can choose $a_j$ and $b_j$ such that   
$s=e^{ix}$ for some $x\in\ ]a,b[\ \cap\ ]a_j,b_j[$. 
We also have $x'$ in $]a,b[$ such 
that $x\leq x'$ and $e^{ix'}=s'$. 
We cannot have $b_j<x'$ otherwise the interval $]x,b_j[$ would be included in $]a,b[$ so we would have $x''\in\ ]a,b[\ \cap\ \rational$ with $h(e^{ix''})$ arbitrarily small. In particular 
$(e^{ix''},t)$ would not belong to $O_h$.  
Moreover we have $b_j\not=x'$ because one is rational while the other is not. Since the standard order on $\real$ is total, we have $x'<b_j$. Hence both $x$ and $x'$ belong to $]a_j,b_j[$, and we have $x\leq x'$ so $s\leq_{\alpha_j}s'$. Moreover $t< h(s)\leq t_j$. It follows that $f(s,t)\leq_Uf(s',t)$.
\end{proof}

\emph{Competing interests: The authors declare none.}

\end{document}